\documentclass[a4paper]{amsart}

\usepackage{amsmath}
\usepackage{amsfonts}
\usepackage{amssymb}
\usepackage{amsthm}
\usepackage[cmtip,all]{xy}

\newcommand{\ZZ}{\mathbb{Z}}
\newcommand{\G}{{\mathbb G}}
\newcommand{\sO}{{\mathcal O}}
\newcommand{\sF}{{\mathcal F}}
\newcommand{\sG}{{\mathcal G}}
\newcommand{\colim}{\operatorname{colim}}

\newcounter{counter}
\theoremstyle{plain} 
\newtheorem{theorem}[counter]{Theorem}
\newtheorem{proposition}[counter]{Proposition}
\newtheorem{lem}[counter]{Lemma}
\newtheorem{cor}[counter]{Corollary}

\begin{document}

\author{Moritz Kerz}
\address[Moritz Kerz]{Fakult{\"a}t f{\"u}r Mathematik, Universit{\"a}t Regensburg, 93040 Regensburg, Germany, \textit{e-mail address:} \texttt{moritz.kerz@mathematik.uni-regensburg.de}}

\author{Florian Strunk}
\address[Florian Strunk]{Fakult{\"a}t f{\"u}r Mathematik, Universit{\"a}t Regensburg, 93040 Regensburg, Germany, \textit{e-mail address:} \texttt{florian.strunk@mathematik.uni-regensburg.de}}

\thanks{The authors are supported by the SFB/CRC 1085 \emph{Higher Invariants} (Universit\"at Regensburg) funded by the DFG}

\title[Vanishing of $KH$-theory]{On the vanishing of negative homotopy $K$-theory}

\begin{abstract}
We show that the homotopy invariant algebraic $K$-theory of Weibel vanishes below the negative of
the Krull dimension of a noetherian scheme. This gives evidence for a conjecture of Weibel
about vanishing of negative algebraic $K$-groups.
\end{abstract}

\maketitle

\begin{quote}
\em To Chuck Weibel on the occasion of his 65th birthday.
\end{quote}

\medskip

\section{Introduction}

The aim of this note is to prove the following theorem.  For a scheme $X$ and $i\in
\mathbb Z$ we consider homotopy
$K$-theory $KH_i(X)$ as defined in \cite[Sec.~IV.12]{KBook}.

\begin{theorem}\label{mainthm}
Let $X$ be a noetherian scheme of finite Krull dimension $d$. Then $KH_i(X)=0$ for $i<-d$.
\end{theorem}

Weibel conjectured the analogue of this theorem with the
$K$-Theory of Bass--Thomason--Trobaugh in place of homotopy $K$-theory, originally
formulated as a question in~\cite[Qu.~2.9]{Weibel}. In fact Theorem~\ref{mainthm} is a special
case of Weibel's original
conjecture, as can be seen using the spectral sequence~\eqref{spseq}. 

\begin{cor}\label{cor1}
Let $X$ be a noetherian scheme of finite Krull dimension $d$ and let $p$ be a prime nilpotent on
$X$. Then $K_i(X)\otimes_\ZZ \ZZ[1/p]=0$ for $i<-d$.
\end{cor}

Note that under the conditions of Corollary~\ref{cor1}  we have
\[
KH_i(X)\otimes_\ZZ \ZZ[1/p]\cong K_i(X)\otimes_\ZZ \ZZ[1/p] 
\]
according to a result of Weibel, see \cite[Thm.~9.6]{TT}, where $K$ denotes the Bass--Thomason--Trobaugh $K$-theory.

Corollary~\ref{cor1} has been shown in \cite{Kelly} for $X$ quasi-excellent using the alteration
theorem of Gabber--de~Jong. Our proof is more elementary as instead of weak resolution of
singularities we use {\it platification par \'eclatement} \cite[Thm.~5.2.2]{RG}.

\section{Some reductions}

\begin{proposition}\label{proploc}
Let $X$ be a noetherian scheme of finite Krull dimension $d$. If
$KH_i(\sO_{X,x})=0$ for all $x\in X$ and $i<-\dim(\sO_{X,x})$ then $KH_i(X)=0$ for $i<-d$.
\end{proposition}

In the proof of Proposition~\ref{proploc} we need the following classical result of Grothendieck.

\begin{lem}\label{lem.vani} Let $r\ge 0$ be an integer. 
Let $\mathcal F$ be a Zariski sheaf on the noetherian scheme $X$. Assume
that $\mathcal F_x=0$ for all points $x\in X$  with $\dim \overline{ \{ x \}}>r$. Then
$H^i(X,\mathcal F)=0$ for $i>r$.
\end{lem}

\begin{proof}
Let $J=\coprod_{U\subseteq X} \mathcal F(U)$, where $U$ runs through all open subsets of
$X$, and let $I$ be the set of finite subsets of $J$. For $\alpha\in I$ let $\mathcal
F_\alpha$ be the abelian subsheaf of $\mathcal F$ locally generated by the sections in
$\alpha$. Then 
\[
\mathcal F= \colim_{\alpha\in I} \mathcal F_\alpha 
\]
as a filtered colimit. Since each local section of $\sF$ is supported on a closed subscheme
of dimension at most $r$ there are closed immersions $i_\alpha: X_\alpha \to X$ and
abelian sheaves $\sG_\alpha$ on $X_\alpha$ such that $\dim X_\alpha \le r$ and such that
$i_{\alpha,*}(\sG_\alpha)\cong \sF_\alpha$. Then
\[
H^i(X,\sF) \stackrel{\bf (1)}{=}\colim_{\alpha\in I} H^i(X,\sF_\alpha)  = \colim_{\alpha\in I}
H^i(X_\alpha,\sG_\alpha) \stackrel{\bf (2)}{=} 0
\]
for $i>r$.
Here $\bf (1)$ is due to \cite[Prop.~III.2.9]{Hart} and $\bf (2)$ is due to   \cite[Thm.~III.2.7]{Hart}.
\end{proof}

\begin{proof}[Proof of Proposition~\ref{proploc}]

 Consider the convergent Zariski-descent spectral sequence, analogous to \cite[Thm.~10.3]{TT},
\begin{equation}\label{eq.desc}
E_2^{p,q} = H^p(X,\mathcal{KH}_{-q,X}) \Rightarrow KH_{-p-q}(X),
\end{equation}
where $\mathcal{KH}_{i,X}$ is the Zariski sheaf on $X$ associated with $KH_i$.
For $i<-d$ and for $-p-q=i$ let $\sF$ be $\mathcal{KH}_{-q,X}$ and let $r$ be $d-q$.
Then under the conditions of Proposition~\ref{proploc} we get $\sF_x=KH_{-q}(\sO_{X,x})=0$ for all $x\in X$ with $\dim \overline{ \{
  x \}}>r$ since
\[\dim \sO_{X,x}\le \dim X - \dim \overline{ \{ x \}} <\dim X -r =q .\]
So by Lemma~\ref{lem.vani} we deduce $E^{p,q}_2=H^p(X,\sF)=0$ for all $-p-q=i<-d$ and therefore also $KH_i(X)=0$.
\end{proof}

\medskip

The following proposition is immediate in case the scheme $X$ has a desingularization. However, we
avoid any assumption on the existence of resolution of singularities by using
Raynaud--Gruson's platification par \'eclatement instead.

\begin{proposition}\label{prop}
Let $X$ be a reduced scheme which is quasi-projective over a noetherian ring. Let $f\colon
Y\to X$ a smooth and quasi-projective
morphism. Let $k>0$ be an integer and let $\xi\in K_{-k}(Y)$. There exists a birational
projective morphism $p\colon X'\to X$ such that $\tilde p^*(\xi)=0\in K_{-k}(Y')$ where $\tilde p\colon Y' \to Y$ is the pull-back of $p$ along $f$.
\end{proposition}

\begin{proof}

By Bass's definition of negative $K$-theory \cite[Sec.~III.4]{KBook} the group $K_{-k}(Y)$ for
$k>0$ is a quotient of
$K_0(Y\times \mathbb G_m^k) $, where $\G_m=\mathbb A^1\setminus \{ 0 \}$. 
Elements of this $K_0$-group coming from $K_0(Y\times \mathbb A^k) $ vanish in
$K_{-k}(Y)$. 

 Without loss
of generality $\xi$ is represented by a vector bundle $V$ on $Y\times \mathbb G_m^k$. We
can extend $V$ to a coherent sheaf $\bar V$ on $Y\times \mathbb A^k $, see
\cite[Sec.~I.9.4]{EGA}. Choose an open dense subscheme $U\subseteq X$ such that $\bar V$
is flat over $U$. This is possible as $X$ is reduced \cite[Thm.~IV.11.1.1]{EGA}.
By {\it platification par \'eclatement} \cite[Thm.~5.2.2]{RG} there is a projective birational morphism $p:X'\to
X$ which is an isomorphism over $U$ and such that the strict transform $\bar V'=p^{\rm st}(\bar
V)$ as a coherent sheaf on $Y'\times \mathbb A^k$ is flat
over $X'$,   here $Y'=X'\times_X Y$.

Recall that the strict transform $p^{\rm st}(\bar
V)$ is defined as the image of $\hat p^{*}(\bar V)\to j_* j^* \hat p^{*}(\bar V)$, where
$j:f^{-1}(U)\times \mathbb A^k \to Y'\times\mathbb A^k $
is the canonical open immersion and where $\hat p$ denotes  the induced morphism $Y'\times
\mathbb A^k \to Y\times \mathbb A^k$.

Note that  $\bar V'|_{Y'\times \G_m^k}$ is isomorphic to the usual pull-back of the
sheaf $V$, as the latter is flat over $X$.

\begin{lem}\label{tordim} $\bar V'$ has finite Tor-dimension as an $\sO_{Y'\times \mathbb
    A^k}$-sheaf.
\end{lem}

Lemma~\ref{tordim} implies by \cite[Prop.~II.8.3.1]{KBook} that $\bar V'$ induces an element of $K_0(Y'\times \mathbb
A^k)$ whose restriction to $Y'\times \G_m^k$
 represents  $\tilde p^*(\xi)\in K_{-k}(Y')$ via the Bass construction explained above. As any
such element in negative $K$-theory vanishes we have proved Proposition~\ref{prop}.
\end{proof}

\begin{proof}[Proof of Lemma \ref{tordim}]
For a noetherian scheme $Z$ we denote by $D(Z)$ the derived category of $\sO_Z$-modules  whose cohomology
sheaves are quasi-coherent
and by $D^b(Z)$ the triangulated subcategory of bounded complexes with coherent cohomology sheaves.
Let $y\in Y'\times  \mathbb A^k$ be a point with image $x\in X'$. Let $i_y:y\to Y'\times
\mathbb A^k$ be the natural map, $i_x:F_x\to Y'\times  \mathbb A^k $ the inclusion of the fiber $F_x$ of $Y'\times
\mathbb A^k\to X'$ over $x$ and let $i^x_y:y\to F_x$ be the canonical morphism. By \cite[Prop.~4.4.11]{WeibelHomological} we have to show that $L i^*_y(\bar
V')\in D(y)$ lies in $ D^b(y)$. As $\bar V'$ is flat over $X'$ we have $L i^*_x (\bar V')=
i^*_x(\bar V')\in D^b(F_x)$. As $F_x$ is a regular scheme, $L  (i^x_y)^*$ maps $D^b(F_x)$ to
$D^b(x)$ \cite[Thm.~4.4.16]{WeibelHomological}, so $L i^*_y(\bar V')= ( L  (i^x_y)^* \circ L i^*_x) (\bar V')$ lies in $D^b(y)$.
\end{proof}

\section{Proof of Theorem~\ref{mainthm}}

In the proof of Theorem~\ref{mainthm} we can, using Proposition~\ref{proploc}, restrict to schemes $X$ which are quasi-projective
over noetherian rings. For such $X$ we 
argue inductively on the dimension $d=\dim(X)$. We may assume that $X$ is reduced as
$KH_i(X)=KH_i(X_{\rm red})$, use \cite[Thm.~2.3]{WeibelKH} and Zariski-descent.
The case $d=0$ of Theorem~\ref{mainthm} is shown in \cite[Prop.~3.1]{WeibelKH}.

Let $d>0$ and assume Theorem \ref{mainthm} for all schemes of Krull dimension  less than
$d$ which are quasi-projective over a noetherian ring. Let $\Delta^\bullet$ be the usual cosimplicial scheme defined in degree $p$ by $\Delta^p=\operatorname{Spec}(\ZZ[T_0,\ldots,T_p]/(\sum T_j-1))$. There is a right half-plane spectral sequence
\begin{equation}\label{spseq}
E^1_{p,q}(X)=K_q(X\times \Delta^p) \Rightarrow KH_{p+q}(X),
\end{equation}
functorial in $X$, see~\cite[Prop.~5.17]{Thomason}. This is the
Bousfield--Kan spectral sequence arising from the simplicial spectrum $K(X\times\Delta^\bullet)$ whose homotopy colimit is $KH(X)$ by definition. For each $p+q$ there is a filtration
\begin{align}\label{filtration}
0=F_{-1}(X)\subseteq F_0(X) \subseteq F_1(X) \subseteq \ldots ~ \cup_{p=0}^\infty F_p(X) = KH_{p+q}(X)
\end{align}
with $F_p(X)/F_{p-1}(X) \cong E^\infty_{p,q}(X)$.

Let $i<-d$. In order to conclude that $KH_i(X)= 0$, we show inductively on $p\geq 0$
that the group $F_p(X)$ in the filtration \eqref{filtration} vanishes for all $X$ as above with
$\dim(X)\le d$ at once. Fix a  scheme $X$ of Krull dimension $d$ which is quasi-projective
over a noetherian ring and let $\gamma\in F_p(X)$ be an element. We have $F_p(X)\cong E^\infty_{p,q}(X)$ by the induction hypothesis on $p$. As $E^\infty_{p,q}(X)$ is a subquotient of $E^1_{p,q}(X)$, the element $\gamma$ lifts to a class $\xi\in K_q(X\times \Delta^p)$. Note that $q<-d<0$.

By Proposition \ref{prop} applied to the morphism $Y=X\times\Delta^p\to X$, we find a
projective birational morphism $p\colon X'\to X$ such that $\tilde p^*(\xi)=0\in K_q(Y')$,
here $Y'=X'\times_X Y$. We choose a nowhere dense closed subscheme $Z\hookrightarrow X$ such that $p$ is an isomorphism outside $Z$ and obtain a cdh-distinguished square
\[
\xymatrix{
Z' \ar[r]\ar[d] &X' \ar[d]^p\\
Z \ar[r]        & X.
}
\]
As homotopy $K$-Theory satisfies cdh-descent by \cite[Thm.~3.9]{Cisinski}, we get a long exact sequence
\[
\cdots\to KH_{i+1}(Z') \to KH_i(X)\to  KH_i(Z)\oplus KH_i(X') \to \cdots .
\]
The groups $KH_{i+1}(Z')$ and $KH_i(Z)$ vanish by the induction hypothesis on $d$ as
$\operatorname{dim}(Z'),\operatorname{dim}(Z)<d$, so $KH_i(X)\to KH_i(X')$ is injective (recall that $i<-d$). Hence, it suffices to show that
$p^*\colon KH_i(X)\to KH_i(X')$ maps the element $\gamma$ to zero. Since
$\operatorname{dim}(X')\leq \operatorname{dim}(X)$, we have $F_p(X')\cong
E^\infty_{p,q}(X')$ by the induction hypothesis. The morphism $p^*\colon KH_i(X)\to
KH_i(X')$ restricts to a morphism $F_p(X)\to F_p(X')$ which is compatible with $\tilde p^*\colon
K_q(Y)\to K_q(Y')$. Since $\tilde p^*(\xi)=0\in K_q(Y') $, we conclude
that $p^*(\gamma)=0$, so $\gamma=0$. Hence we obtain $F_p(X)= 0$.


\begin{thebibliography}{10}

\bibitem{Cisinski}
Denis-Charles Cisinski.
\newblock Descente par \'eclatements en {$K$}-th\'eorie invariante par
  homotopie.
\newblock {\em Ann. of Math. (2)}, 177(2):425--448, 2013.

\bibitem{EGA}
Alexandre Grothendieck, and Jean Dieudonn\'e.
\newblock \'{E}l\'ements de g\'eom\'etrie alg\'ebrique. 
\newblock {\em Inst. Hautes \'Etudes Sci. Publ. Math.}, 1960--1967.

\bibitem{Hart}
Robin Hartshorne. 
\newblock
Algebraic geometry.
\newblock 
Graduate Texts in Mathematics, No. 52. Springer-Verlag, New York-Heidelberg, 1977. 

\bibitem{Kelly}
Shane Kelly.
\newblock Vanishing of negative {$K$}-theory in positive characteristic.
\newblock {\em Compos. Math.}, 150(8):1425--1434, 2014.

\bibitem{RG}
Michel Raynaud and Laurent Gruson.
\newblock Crit\`eres de platitude et de projectivit\'e. {T}echniques de
  platification d'un module.
\newblock {\em Invent. Math.}, 13:1--89, 1971.

\bibitem{TT}
R.~W. Thomason and Thomas Trobaugh.
\newblock Higher algebraic {$K$}-theory of schemes and of derived categories.
\newblock In {\em The {G}rothendieck {F}estschrift, {V}ol.\ {III}}, volume~88
  of {\em Progr. Math.}, pages 247--435. Birkh\"auser Boston, Boston, MA, 1990.

\bibitem{Thomason}
Robert~W. Thomason.
\newblock Algebraic {$K$}-theory and \'etale cohomology.
\newblock {\em Ann. Sci. \'Ecole Norm. Sup. (4)}, 18(3):437--552, 1985.

\bibitem{WeibelKH}
Charles~A. Weibel.
\newblock Homotopy algebraic {$K$}-theory.
\newblock In {\em Algebraic {$K$}-theory and algebraic number theory
  ({H}onolulu, {HI}, 1987)}, volume~83 of {\em Contemp. Math.}, pages 461--488.

\bibitem{Weibel}
Charles~A. Weibel.
\newblock {$K$}-theory and analytic isomorphisms.
\newblock {\em Invent. Math.}, 61(2):177--197, 1980.

\bibitem{WeibelHomological}
Charles~A. Weibel.
\newblock {\em An introduction to homological algebra}, volume~38 of {\em
  Cambridge Studies in Advanced Mathematics}.
\newblock Cambridge University Press, Cambridge, 1994.

\bibitem{KBook}
Charles~A. Weibel.
\newblock {\em The {$K$}-book}, volume 145 of {\em Graduate Studies in
  Mathematics}.
\newblock American Mathematical Society, Providence, RI, 2013.
\newblock An introduction to algebraic $K$-theory.

\end{thebibliography}
\end{document}